\documentclass[12pt]{amsart}
\usepackage[margin=1in]{geometry}
\usepackage{amssymb,amsfonts,latexsym,amsmath,amsthm,graphicx}
\usepackage{parskip}
\usepackage{mathrsfs}
\usepackage{mathtools}
\mathtoolsset{showonlyrefs,showmanualtags}

\usepackage[margin=1in]{geometry}

\numberwithin{equation}{section}
\newtheorem{thm}{Theorem}[section]
\newtheorem{prop}[thm]{Proposition}

\newtheorem{lemma}[thm]{Lemma}
\newtheorem{conj}[thm]{Conjecture}
\newtheorem{prob}[thm]{Problem}
\theoremstyle{remark}

\newcommand{\EE}{\mathbb{E}}

\newcommand{\RR}{\mathbb{R}}
\newcommand{\CC}{\mathbb{C}}

\usepackage[colorinlistoftodos]{todonotes}

\newcommand{\ol}{\overline}

\newcommand{\e}{\varepsilon}

\newcommand{\be}{\begin{equation}}
\newcommand{\ee}{\end{equation}}

\begin{document}

\title[Random harmonic polynomials]{On the average number of zeros of random harmonic polynomials with i.i.d.\@ coefficients: precise asymptotics}
\date{}

\author[E. Lundberg, A. Thomack]{Erik Lundberg and Andrew Thomack}

\begin{abstract}
Addressing a problem posed by W. Li and A. Wei (2009), we investigate the average number of (complex) zeros of a random harmonic polynomial $p(z) + \overline{q(z)}$ sampled from the Kac ensemble, i.e.,  
where the coefficients are independent identically distributed centered complex Gaussian random variables.
We establish a precise asymptotic, showing that when $\deg p = \deg q = n$ tends to infinity the average number of zeros is asymptotic to $\frac{1}{2} n \log n$.  We further consider the average number of zeros restricted to various regions in the complex plane leading to interesting comparisons with the classically studied case of analytic Kac polynomials.
We also consider deterministic extremal problems for harmonic polynomials with coefficient constraints; using an indirect probabilistic method we show the existence of harmonic polynomials with unimodular coefficients having at least $\frac{2}{\pi} n \log n + O(n)$ zeros.
We conclude with a list of open problems.
\end{abstract}

\maketitle

\section{Introduction}

Let $H_{n,m}(z) = p(z) + \ol{q(z)}$ denote a (complex) harmonic polynomial with $n= \deg p \ge \deg q = m > 0$.
In 1992, T. Sheil-Small posed the problem of studying the maximum valence (number of zeros) of $H_{n,m}$ and conjectured that for each $n$ the maximum, taken over all $p$ of degree $n$ and $q$ of degree less than $n$, is $n^2$.  A.S. Wilmshurst proved this conjecture in his thesis \cite{W1}, \cite{W2} using Bezout's Theorem to establish the upper bound and constructing examples to show sharpness of the Bezout bound for each $m=n-1$.  Wilmshurst also stated a refinement of the conjecture that for each $n,m$ with $m<n$ the maximum number of zeros, taken over $p$ of degree $n$ and $q$ of degree $m$, increases linearly in $n$ for each fixed $m$, namely, the maximum is $m(m-1) + 3n-2$. 
The $m=n-1$ case of the conjecture follows from Wilmshurst's results. D. Khavinson and G. Swiatek \cite{KhSw} verified the conjectured upper bound $3n-2$ for the case $m=1$ using an indirect technique based on Fatou's theorem from holomorphic dynamics, and L. Geyer \cite{Geyer} used further tools from holomorphic dynamics, namely Thurston equivalence of topological polynomials, to show sharpness of the upper bound $3n-2$.  Despite this progress, the conjecture turned out to be false in general; counterexamples for the case $m=n-3$ were produced in \cite{LLL}.  Further counterexamples were constructed in \cite{HLLM}, \cite{LS}, and an indirect probabilistic method was used in \cite{LundRand} to show the existence of examples for each $n,m$ with at least $\lceil n \sqrt{m} \rceil$ zeros (a similar method is used in Section \ref{sec:existenceproof} below to study a constrained version of Sheil-Small's problem).  This latter result violates Wilmshurst's conjecture for all $m>9$ and shows that, if the true maximum is indeed linear in $n$ for each fixed $m$ then, contrary to Wilmshurst's conjecture, the linear growth rate must depend on $m$.  

Motivation for studying harmonic mappings comes from several sources, including minimal surface theory, where they appear in the Weierstrauss-Enneper parameterization of a minimal surface \cite{Duren}, differential geometry where they appear as energy minimizing maps between Riemannian manifolds \cite{Manifolds}, numerical analysis, where they are used for efficient mesh generation \cite{Smith}, computer animation, where they have been used (by Pixar) for character articulation \cite{Pixar}, and gravitational lensing, where they appear as lensing maps used to model the effect of intervening massive objects on background light sources \cite{Petters}.
In most of these applications, univalent (one-to-one) mappings are desired, but in the latter mentioned application to the theory of gravitational lensing, multi-valent mappings are especially interesting as they give rise to multiple observed images of a single background source.  Multi-valent harmonic maps are of importance in additional areas of mathematical physics, such as in elasticity theory, namely, in the study of critical points of the torsion function which is formally equivalent to a problem in another physical setting concerning localization of eigenfunctions \cite{Torsion}.  The basic reason for the prevalence of harmonic mappings in two-dimensional problems of mathematical physics is potential theoretic; the gradient of a harmonic (real-valued) planar potential is not analytic but rather anti-analytic (i.e., the conjugate of an analytic function), for instance, the physical vector fields modeled in an engineering course on complex variables (planar electrostatic fields, heat flow fields of steady state temperature distributions, and velocity fields of ideal fluids) are anti-analytic \cite{SaffSnider},  \cite{Fisher}.  Additional analytic terms arising from external confining potentials or other physical considerations can then lead to a sum of analytic and anti-analytic expressions, i.e., complex harmonic functions.  Equilibrium problems in a variety of such physical settings then lead to valence problems for harmonic mappings.  In addition to direct applications in mathematical physics, problems on valence of harmonic mappings serve as prototypes for equilibrium enumeration problems in more challenging non-harmonic settings such as Maxwell's problem on equilibria in electrostatic fields generated by point charges \cite{GaNoSh} or enumeration of equilibria in the circular restricted $n$-body problem \cite{BaLe}; see \cite{Newton} for a discussion of these problems and comparison with the harmonic setting of the image counting problem from gravitational lensing.

Sheil-Small's problem on the valence of harmonic polynomials is close to the study of the maximal number of lensed image in gravitational lensing, and an extension \cite{KN} of the above-mentioned result \cite{KhSw} to certain rational harmonic functions (again relying on holomorphic dynamics) confirmed astronomer S. Rhie's conjecture \cite{Rhie} on the maximal number of images gravitationally lensed by a point mass ensemble.  
Motivated by the connection to gravitational lensing the zeros of rational harmonic functions were investigated further in \cite{Bleher}, \cite{SeteCMFT}, \cite{SeteGrav}, \cite{SetePert}, \cite{Zur2018a}, \cite{Zur2018b}), and we also note in passing that 
subsequent adaptations of the indirect method based on holomorphic dynamics have led to solutions to additional problems including sharp estimates for the topology of quadrature domains \cite{LeeMakarov}, classification of the number of critical points of Green's function on a torus \cite{BergErem}, a proof of Bshouty and Hengartner's conjectured bound for the valence of logharmonic polynomials \cite{KLP}, and sharp estimates for the number of solutions of certain transcendental equations \cite{BergErem2018}.
Particularly relevant to the current paper, we note that the probabilistic study of the valence of harmonic polynomials initiated by W. Li and A. Wei \cite{LiWei} is closely related to the study of stochastic gravitational lensing \cite{Wei}, \cite{PettersStochastic1}, \cite{PettersStochastic2} and uses some of the same tools such as the Kac-Rice formula for vector fields.
It is generally useful in applied settings to understand average case behavior in addition to extremal behavior, so the variety of applications mentioned in the previous paragraph serve as broad motivation for the probabilistic study of harmonic mappings--the setting of the current paper.

\subsection{Zeros of random harmonic polynomials with i.i.d.\@ coefficients}

Given the variability in the possible number of zeros, it is natural to consider the \emph{expected} number of zeros when $H_{n,m}$ is random.
This line of inquiry, which was initiated by Li and Wei in \cite{LiWei}, requires specifying a definition for the word ``random''.  One obvious choice is to sample $p$ and $q$ independently with i.i.d.\@ (independent identically distributed) coefficients, and it is natural to consider standard complex Gaussians.
Hence, we will consider the (harmonic) Kac ensemble where the harmonic polynomial $H(z) = H_{n,m}(z) = p_n(z) + \ol{q_m(z)}$ is randomized by sampling $p_n, q_m$ independently from the analytic Kac ensemble, i.e.,
\be\label{eq:harmonicKac}
H(z) = p_n(z) + \ol{q_m(z)}, \quad p_n(z) = \sum_{k=0}^n A_k z^k, \quad q_m(z) = \sum_{k=0}^m B_k z^k,
\ee
where $A_k,B_k \sim N_{\CC}(0,1)$ are i.i.d.\@ standard complex Gaussians.  We recall that a standard complex Gaussian $Z \sim N_{\CC}(0,1)$ can be expressed as a combination $Z=A+iB, A,B \sim N_{\RR}(0,1/2)$ of real independent Gaussians $A,B$ of mean zero and variance $1/2$.

This case (the harmonic Kac model) was considered in \cite{LiWei} where the average number of zeros in an open set $T$ is expressed as the following integral of a deterministic function (depending on $n$) over $T$.
\begin{thm}[Li, Wei \cite{LiWei}]\label{thm:LiWei}
The expectation $\EE N_H(T)$ of the number of zeros in an open set $T \subseteq \CC$ of a random harmonic polynomial $H(z) = p_n(z)+\ol{q_m(z)}$ sampled from the Kac model satisfies
\begin{equation}
\label{eq:LiWei}
\mathbb{E}N_H(T)=\frac{1}{\pi}\int_T\frac{1}{\lvert z\rvert^2}\frac{r_1^2+r_2^2-2r_{12}^2}{r_3^2\sqrt{(r_1+r_2)^2-4r_{12}^2}}dA(z),
\end{equation}where $dA(z)$ denotes the Lebesgue measure on the plane, and
\begin{align*}
r_3=&\sum_{j=0}^n \lvert z\rvert^{2j}+\sum_{j=0}^m \lvert z\rvert^{2j},\qquad\qquad\quad r_{12}=\left(\sum_{j=1}^nj \lvert z\rvert^{2j}\right)\left(\sum_{j=1}^mj \lvert z\rvert^{2j}\right),
\\r_1=&r_3\sum_{j=1}^nj^2 \lvert z\rvert^{2j}-\left(\sum_{j=1}^nj \lvert z\rvert^{2j}\right)^2,
\ \ r_2=r_3\sum_{j=1}^mj^2 \lvert z\rvert^{2j}-\left(\sum_{j=1}^mj \lvert z\rvert^{2j}\right)^2.
\end{align*}
\end{thm}
This result is proved using a version of the Kac-Rice formula for vector fields which gives
\begin{equation}\label{eq:KacRiceVec}
\mathbb{E}N_H(T)=\int_T\mathbb{E}\left(\lvert\det J_H(z)\rvert\big\vert H(z)=0\right)\rho(0;z)dA(z),
\end{equation}
where, for each $z$, $\rho(u;z)$ is the probability density function of $u=H(z)$.
Theorem \ref{thm:LiWei} then follows from an explicit computation of the expectation appearing inside the integral on the right hand side in \eqref{eq:KacRiceVec}.  We include a modified version of Li and Wei's proof in Section \ref{sec:prelim} below.

The asymptotic analysis of \eqref{eq:LiWei} was left as an open problem (as we discuss in Section \ref{sec:models} below, the authors of \cite{LiWei} did obtain precise asymptotics for a different model referred to as the Kostlan model).  For the case when $m$ is fixed and $n \rightarrow \infty$, Li and Wei conjectured that the average number of zeros for the Kac model is asymptotically $n$.  This was confirmed by the second named author in his thesis \cite{AndyThesis}
where it was further conjectured that when $m \sim n$ as $n \rightarrow \infty$ the expected number of zeros satisfies $\EE N_H \sim C n \log n$ for some constant $C>0$.

\begin{conj}[\cite{AndyThesis}]
With $p_n,q_n$ sampled independently from the Kac ensemble, the expected number of zeros of $H(z) = p_n(z) + \ol{q_n(z)}$ satisfies $\EE N_H \sim C n \log n$ as $n \rightarrow \infty$ for some constant $C>0$.
\end{conj}

Our main result, stated in the following theorem, provides the precise asymptotic confirming this conjecture while also giving an explicit value of the constant $C=1/2$.

\begin{thm}\label{thm:iid}
Let $H(z) = p_n(z) + \ol{q_n(z)}$ be a random harmonic polynomial with i.i.d.\@ Gaussian coefficients.
Then the expectation of the number of zeros of $H$ satisfies
$$\EE N_{H}(\CC) \sim  \frac{1}{2}  n \log n, \quad \text{as } n \rightarrow \infty.$$
\end{thm}

The proof of this result, which is based on asymptotic analysis of Li and Wei's Kac-Rice type integral \eqref{eq:LiWei}, gives additional information on the average number of zeros appearing in various regions of the complex plane.  Namely, Theorem \ref{thm:iid} is a consequence of the following more refined result.

\begin{thm}
\label{thm:refined}
Let $H(z) = p_n(z) + \ol{q_n(z)}$ be a random harmonic polynomial with i.i.d.\@ Gaussian coefficients.
Then the expectation of the number of zeros of $H$ satisfies the following asymptotic estimates as $n \rightarrow \infty$
\be
\EE N_{H}(\{|z|^2 \leq 1+\tfrac{(\log n)^2}{n} \} )  = O(n \log \log n),
\ee
\be
\EE N_{H}(\{1+\tfrac{(\log n)^2}{n}  \leq |z|^2 \leq 1+ (\log n)^{-1} \} )  \sim \frac{1}{2} n \log n,
\ee
\be
\EE N_{H}(\{|z|^2 \geq 1+ (\log n)^{-1} \} )  = O(n \log \log n).
\ee
\end{thm}

As a particular consequence of Theorem \ref{thm:refined}, we note that the entire leading order contribution to the average number of zeros comes from an annulus of shrinking width located in the exterior of the unit disk and converging to the unit circle.
We compare and contrast this outcome with the well-studied case of analytic polynomials in the next subsection.

\subsection{Comparison with zeros of random analytic polynomials}

The zeros of random analytic polynomials with i.i.d.\@ coefficients have been studied extensively \cite{Shepp}, \cite{KabZap}, \cite{KoushikIgor}, \cite{Zelditch}, \cite{Zeitouni}; while the total number of zeros in the analytic case is determined by the degree, it is interesting to study the location of zeros which are known to approximately equidistribute on the unit circle.  To be precise, the empirical measure (normalized counting measure) associated to the zero set converges (as the degree increases) weakly almost surely to the uniform measure on the unit circle (this result has been extended beyond the Gaussian setting under a mild tail decay assumption on the distribution from which the coefficients are sampled \cite{KoushikIgor}).
It seems likely that such an equidistribution result also holds for the harmonic case considered above, and this would be consistent with the statement in Theorem \ref{thm:refined}.

From another aspect of the statement of Theorem \ref{thm:refined} we notice a sharp contrast between the distribution of zeros for the harmonic and analytic cases.
Namely, from the statement of Theorem \ref{thm:refined} we see that the expected number of zeros inside the unit disk represents an asymptotically vanishing portion of the zeros. 
In the analytic case, the expected number of zeros inside the unit disk is equal to the expected number of zeros in the exterior of the unit disk. This symmetry for the analytic case is easy to see by starting with the given analytic Kac polynomial $f(z)$ and considering the equivalent (in distribution) transformed polynomial $z^{n} f(z^{-1})$.  This transformation does not preserve harmonic Kac polynomials as it does not even preserve harmonicity.
Instead we have an appearance of highly oscillatory factors as we explain below.  Rather than $z^n$, let us multiply outside by $r^n$ where $z=r e^{i \theta}$ as this purely radial factor will give equal preference to the analytic and anti-analytic terms.
Let us denote $G(z):=r^n H(z^{-1}) = r^n p(z^{-1}) + r^n \ol{q(z^{-1})}$.
We have $r^n p(z^{-1}) = e^{-i n \theta}\hat{p}(z)$, where $\hat{p}$ is distributed as a Kac polynomial, and we have $r^n \ol{q(z^{-1})} = e^{i n \theta} \ol{\hat{q}(z)}$ is the complex conjugate of a Kac polynomial $\hat{q}$ multiplied by $e^{i n\theta}$, where $z=r e^{i \theta}$, i.e., $\theta$ denotes the argument of $z$.
As the degree $n$ tends to infinity, the Kac polynomials $p$, $q$, $\hat{p}$, $\hat{q}$ each tend toward Gaussian power series that converge almost surely in the unit disk, and in this sense they ``stabilize'' in the unit disk (there is of course still randomness in the limiting series, but no longer dependence on degree).  On the other hand, rather than stabilizing the function $G(z) = e^{-i n\theta} \hat{p}(z) + e^{i n \theta} \ol{\hat{q}(z)}$ has the highly oscillatory (deterministic) factors $e^{\pm i n\theta}$, potentially leading to many more zeros in the unit disk on average than $H(z)$.  (Recall that zeros of $G(z)$ inside the unit disk correspond to zeros of $H(z)$ outside the unit disk.)

The above heuristic gives some insight as to why there are more zeros of $H$ on average in the exterior of the disk than in the interior, and it seems particularly useful in understanding the behavior of $G$ (and thereby that of $H$) away from the unit circle.
Let us fix a radius $r_0<1$, and consider the increment of the argument of $G(z)$ along the circle of radius $r_0$.
When $n$ is large we estimate that the increment of the argument of $G(z)$ is approximately $-(m_+ - m_-)n = (2\pi -2m_+)n$ where $m_+$ denotes the arc length measure of the subset where $|p(z)|>|q(z)|$ and $m_-$ denotes the arclength measure of the subset where $|p(z)|<|q(z)|$.
The number $m_+$ is random, but ``stabilizes'' with $\hat{p}, \hat{q}$ as the degree $n$ tends to infinity.
Based on the generalized argument principle \cite[Thm. 2.2]{ST} for harmonic functions (while $G$ is not harmonic the increment of its argument is related to that of the harmonic function $H$ along the circle of reciprocal radius $1/r_0$), this suggests there are order-$n$ many zeros of $G$ inside the circle of radius $r_0$ or within an annulus of inner and outer radii less than one, and a corresponding statement holds for $H$ in an annulus with radii larger than one.

While the above argument is not rigorous, we confirm the predicted outcome in the following result concerning regions away from the unit circle, comparing the average number of zeros restricted to regions $U,V$ compactly contained in $|z|>1$ and $|z|<1$ respectively.


\begin{thm}\label{thm:UV}
For any open set $U$ whose closure is contained in the exterior of the closed unit disk, there exists a constant $C_U > 0$ such that
$$\EE N_H(U) \sim C_U \cdot n, \quad (n \rightarrow \infty). $$
For any open set $V$ whose closure is contained in the open unit disk, there exists a constant $C_V > 0$ such that
$$\EE N_H(V) \sim C_V, \quad (n \rightarrow \infty). $$
\end{thm}

In contrast, for an analytic Kac polynomial the average number of zeros in either region $U$ or $V$ (as described in the theorem) approaches a constant (and it follows from the symmetry discussed above that for regions $U,V$ related by inversion with respect to the circle, the constant is the same).

It seems useful to devise a proof of Theorem \ref{thm:UV} based on the above intuitive geometric reasoning involving the oscillatory factors  $e^{\pm i n \theta}$. In the proof given below (in Section \ref{sec:main}) we use a different strategy returning to analysis of the integral \eqref{eq:LiWei} used in the proof of Theorem \ref{thm:refined}.  

\subsection{Other models of random harmonic polynomials}\label{sec:models}

Several other models of random harmonic polynomials, where the coefficients are independent (but not identically distributed) Gaussians, have been considered in \cite{LiWei},  \cite{Lerariotruncated}, \cite{Andy},  \cite{AndyThesis},  \cite{AndyZach}.  Let us briefly summarize those results.

In addition to considering the case of i.i.d.\@ coefficients, Li and Wei \cite{LiWei} considered the case when the coefficients have binomial variance, i.e.\@ $p$ and $q$ are independently sampled from the so-called Kostlan ensemble.
They showed that the expected number of zeros is asymptotically $n^{3/2}$ when $m \sim n$ and asymptotically $n$ when $m$ is fixed or when $m \sim \alpha n$ with $0 < \alpha < 1$.

With A. Lerario, the first author considered in \cite{Lerariotruncated} a modified version of the Kostlan model, which they referred to as the ``truncated model'', where the binomial variances for $q$ are taken to be $\binom{n}{k}$ rather than $\binom{m}{k}$ so that $q$ is an independent truncated copy of $p$.  This leads to asymptotically $c_\alpha n^{3/2}$ zeros on average when $m \sim \alpha n$ for $0< \alpha \leq 1$ where $\displaystyle c_\alpha = \tfrac{1}{2} \left[ \arctan \left( \sqrt{\alpha/(1-\alpha)} \right) - \sqrt{\alpha(1-\alpha)} \right]$
and asymptotically $n$ zeros on average when $m$ is fixed and $n \rightarrow \infty$.

With Z. Tyree, the second author considered in \cite{AndyZach} the case when $p$ and $q$ are each so-called Weyl polynomials (also referred to as ``flat'' polynomials).  They showed that the average number of zeros is again asymptotically proportional to $n^{3/2}$ when $m \sim \alpha n$ but with a different leading coefficient.  Namely, there are asymptotically $\frac{1}{3} m^{3/2} \sim \frac{\alpha^{3/2}}{3} n^{3/2}$ zeros on average.

In the work \cite{LundRand} mentioned above that established the existence of counterexamples to Wilmshurst's conjecture over a broad range of degrees $m,n$, the model was selected in an ad hoc way, sampling $q$ from the Kostlan ensemble and taking $p$ to be a perturbation of $q$.  A similar technique will be used in the current paper to prove Theorem \ref{thm:existence} stated below.

Fixing attention on $m \sim n$, the reader may notice that the Kac model of random harmonic polynomials that we focus on in this paper has fewer zeros on average than the other models that have been studied.  This may suggest that it is less suitable for gaining insight on Sheil-Small's extremal valence problem, but as we will see next in Section \ref{sec:constraint} it leads us to some insights on certain constrained versions of Sheil-Small's problem.

\subsection{Extremal valence under coefficient constraints}\label{sec:constraint}
Returning to the deterministic setting, and while motivated by classical extremal problems \cite{ErdelyiProblems} for analytic polynomials with ``Littlewood-type constraints'', let us consider the following restricted versions of Sheil-Small's problem.

\begin{prob}[unimodular polynomials]\label{prob:unimodular}
    For $n>m$, determine the maximum number of zeros of a harmonic polynomial $H_{n,m}(z) = p_n(z) + \ol{q_m(z)}$ having unimodular coefficients.
\end{prob}

\begin{prob}[Littlewood polynomials]\label{prob:pm1}
    For $n>m$, determine the maximum number of zeros of a harmonic polynomial $H_{n,m}(z) = p_n(z) + \ol{q_m(z)}$ where $p_n$ and $q_m$ are Littlewood polynomials, i.e., polynomials whose coefficients are $\pm 1$.
\end{prob}

Let us consider the case $m=n-1$ where the original (unconstrained) Sheil-Small problem is completely resolved. We note that Wilmshurst's extremal examples do not satisfy the unimodular constraint, and it does not seem possible to modify those examples to make them unimodular while preserving the number of zeros.  Hence, the Bezout bound may not be sharp in the restricted context of Problems \ref{prob:unimodular} and \ref{prob:pm1}, and these constrained versions of Sheil-Small's problem appear to be wide open.
The probabilistic perspective gives some insight here, namely, the following conjectured generalization of Theorem \ref{thm:iid} suggests that the maximal valence grows at least at the rate $n \log n$, which is confirmed for the unimodular case in Theorem \ref{thm:existence} below.

\begin{conj}\label{conj:iid}
The conclusion of Theorem \ref{thm:iid} holds when the coefficients are i.i.d.\@ random variables of mean zero, positive
 variance, and having finite moments.
\end{conj}

The following result, which is proved using an indirect probabilistic method, provides some initial insight on the unimodular problem by showing the existence of examples whose number of zeros grows more quickly than linearly in $n$.

\begin{thm}\label{thm:existence}
There exist harmonic polynomials $H(z) = p_n(z) + \ol{q_{n-1}(z)}$ with unimodular coefficients such that the number of zeros of $H$ is at least $\frac{2}{\pi} n \log n + O(n)$ as $n \rightarrow \infty$.
\end{thm}

The proof of Theorem \ref{thm:existence} uses an indirect probabilistic method that hinges on an estimate for the expected number of zeros of a harmonic polynomial where $q_{n-1}$ is sampled with i.i.d.\@ unimodular coefficients and $p_n(z)$ takes the special form $p_n(z)= \pm z^n + q_{n-1}(z)$.  For polynomials of this special form,  we have the following estimate that improves substantially the Bezout bound for large $n$.

\begin{thm}\label{thm:UB}
Let $p(z) = z^n + q(z)$, where 
$q(z)$ is a polynomial of degree $m<n$ with unimodular coefficients, and $q(0)=1$.
Then there exists $C>0$ such that the harmonic polynomial $p(z) + \ol{q(z)}$ has at most $C \cdot n^{3/2}$ zeros.
\end{thm}

The proof of this result is nonprobabilistic and relies on an estimate \cite[Thm. 4.1]{BEK} of Borwein, Erd\'elyi, and K\'oz for the number of real zeros of a univariate (analytic) polynomial under an appropriate constraint on the size of the coefficients.
Their result provides a bound of order $\sqrt{n}$ on the number of real zeros for polynomials satisfying a bound from below on the moduli of the constant coefficient and leading coefficient along with a bound from above on the moluli of the remaining coefficients (as noted in \cite{BEK} a slightly weaker estimate with an additional logarithmic factor follows from the celebrated results of Erd\"os and Tur\'an on the angular distribution of zeros in the complex plane \cite{ErdosTuran}).

\subsection*{Outline of the paper}
In Section \ref{sec:prelim}, we give an alternative proof of Li and Wei's Theorem \ref{thm:LiWei} that serves as an important preliminary result.
This will make the paper more self-contained and also gives us an opportunity to demystify Li and Wei's original proof by avoiding a certain clever step where they rewrite the absolute value function as an integral in order to deal with it indirectly.  We instead compute the joint density of $H(z)$ and its Jacobian using the method of characteristic functions (Fourier transforms) and a well-known result for Hermitian forms in Gaussian variables \cite{Turin60}, after which we are able to deal with the absolute value directly.
In Section \ref{sec:main}, we prove Theorems \ref{thm:refined} and \ref{thm:UV} by providing the necessary asymptotic analysis of the integral in \eqref{eq:LiWei}.  In polar coordinates the integrand is independent of the angular coordinate, reducing this to a one-dimensional integral.  The asymptotic analysis is yet delicate, and we will need to separate the interval of integration into several subintervals where different strategies are used (as a byproduct of this ad hoc approach one may, if desired, extract from the proofs additional information providing estimates for the average number of zeros over a more refined partition of annuli than appears in Theorem \ref{thm:refined}).  The proofs of the deterministic results Theorems \ref{thm:existence} and \ref{thm:UB} are provided in Section \ref{sec:existenceproof}.
We make some concluding remarks in Section \ref{sec:concl}, including discussion of open problems and potential future research directions.

\section{Proof of Theorem \ref{thm:LiWei} (Li and Wei's Kac-Rice type integral formula)}\label{sec:prelim}


The following is a generalization and clarification of Li and Wei's Theorem \ref{thm:LiWei}, giving the expectation of the number of zeros for a harmonic polynomial with independent complex Gaussian coefficients.
\begin{thm}\label{ThmGenLiWei}
The expectation $\mathbb{E}N_H(T)$ of the number of zeros of 
\be
H_{n,m}(z)=\sum_{j=0}^nA_jz^j+\sum_{j=0}^mB_j\overline{z}^j
\ee
where $A_0,\ldots, A_n$ and $B_0,\ldots,B_m$ are mutually independent complex Gaussian random variables with $\mathbb{E}A_j=\mathbb{E}B_j=\mathbb{E}[A_j^2]=\mathbb{E}[B_j^2]=0$ and $\mathbb{E}A_j\overline{A}_j=\alpha_j$ and $\mathbb{E}B_j\overline{B}_j=\beta_j$ on a domain $T\subset\mathbb{C}$ is given by:
\begin{equation}
\label{LiWeiGeneral}
\mathbb{E}N_F(T)=\frac{1}{\pi}\int_T\frac{1}{\lvert z\rvert^2}\frac{r_1^2+r_2^2-2r_{12}^2}{r_3^2\sqrt{(r_1+r_2)^2-4r_{12}^2}}dA(z),
\end{equation}where $dA(z)$ denotes the Lebesgue measure on the plane, and
\begin{align*}
r_3=&\sum_{j=0}^n\alpha_j\lvert z\rvert^{2j}+\sum_{j=0}^m\beta_j\lvert z\rvert^{2j}\qquad\qquad\quad r_{12}=\left(\sum_{j=1}^nj\alpha_j\lvert z\rvert^{2j}\right)\left(\sum_{j=1}^mj\beta_j\lvert z\rvert^{2j}\right)
\\r_1=&r_3\sum_{j=1}^nj^2\alpha_j\lvert z\rvert^{2j}-\left(\sum_{j=1}^nj\alpha_j\lvert z\rvert^{2j}\right)^2
\ \ r_2=r_3\sum_{j=1}^mj^2\beta_j\lvert z\rvert^{2j}-\left(\sum_{j=1}^mj\beta_j\lvert z\rvert^{2j}\right)^2
\end{align*}
\end{thm}
The result in \cite{LiWei} was originally stated for $\alpha_j=\binom{n}{j}$ and $\beta_j=\binom{m}{j}$ and also for $\alpha_j=\beta_j=1$, but as pointed out in \cite{Lerariotruncated} the above more general result follows from their proof.  We also explicitly include the condition $\mathbb{E}[A_j^2]=\mathbb{E}[B_j^2]=0$ that the pseudocovariance of each random variable is $0$, a condition that was not stated in \cite{LiWei} but is implicit when they assume independence of the real and imaginary parts of complex Gaussians.

The following proof uses some essential elements from the original proof \cite{LiWei}, but follows a different route using a result on Conditional Complex Gaussian vectors from \cite{AndyThesis} to avoid moving to real random variables and a theorem from \cite{Turin60} on the characteristic function (Fourier transform) of a Hermitian form in Gaussian variables in order to avoid using an integral formula for the absolute value.  In the proof below, for a matrix $M$ with complex entries, we use $M'$ to denote its transpose and $M^*$ to denote its conjugate transpose.
\begin{proof}
We define
\begin{equation*}P(z)=\sum_{j=0}^nA_jz^j\qquad Q(z)=\sum_{j=0}^mB_jz^j\end{equation*}
We use the Kac-Rice formula as stated in \cite[Lemma 2.1]{LiWei}:
\begin{equation}\label{KacRice}
\mathbb{E}N_H(T)=\int_T\mathbb{E}\left(\lvert\det J_H(z)\rvert\big\vert H(z)=0\right)\rho(0;z)dA(z),
\end{equation}
where, for each $z$, $\rho(u;z)$ is the probability density function of $u=H(z)$.

The modulus of the Jacobian determinant of $H(z)=P(z)+Q(\overline{z})$ is given by
\begin{equation*}\lvert J_H(z)\rvert=\left\lvert\left\lvert P'(z)\right\rvert^2-\left\lvert Q'(\overline{z})\right\rvert^2\right\rvert=\frac{1}{\lvert z\rvert^2}\left\lvert\left\lvert zP'(z)\right\rvert^2-\left\lvert \overline{z}Q'(\overline{z})\right\rvert^2\right\rvert\end{equation*}
To compute the covariance matrix of the (column) vector $[zP'(z)\ \overline{z}Q'(\overline{z})]'$ conditioned on $H(z)=0$ we use the following result from \cite[Cor. 21]{AndyThesis} whose proof is based on elementary linear algebra.
\begin{thm}
Let $\mathbf{Z}_1$ and $\mathbf{Z}_2$ be two complex Gaussian vectors such that the appended vector $\tiny\begin{bmatrix}\mathbf{Z}_1\\\mathbf{Z}_2\end{bmatrix}$ is a complex Gaussian vector. Let $\mathbf{Z}_1$ have mean $\mu_1$, covariance $\Gamma_1$, and pseudocovariance $\EE [\mathbf{Z}_1 \mathbf{Z}_1'] = 0$, and similarly for $\mathbf{Z}_2$, such that $\mathbb{E}[\mathbf{Z}_1\mathbf{Z}_2']=0$.  Given a fixed vector, $\mathbf{z}_2\in\mathbb{C}^n$, the random variable $\mathbf{Z}_1\big|_{\mathbf{Z}_2=\mathbf{z}_2}$ 
 is Gaussian with mean
\begin{equation*}\mu=\mu_1+\Gamma_{12}\Gamma_2^{-1}(\mathbf{z}_2-\mu_2),\end{equation*}
covariance matrix
\begin{equation*}\Gamma=\Gamma_1-\Gamma_{12}\Gamma_2^{-1}\Gamma_{12}^*,\end{equation*}
and pseudocovariance matrix of $0$.  Here, $\Gamma_{12}=\mathbb{E}[\mathbf{Z}_1\mathbf{Z}_2^*]$.
\end{thm}

Applying this result to the case at hand, we set  $\mathbf{Z}_1=\tiny\begin{bmatrix}zP'(z)\\\overline{z}Q'(\overline{z})\end{bmatrix}$ and $\mathbf{Z}_2=H(z)$ while $\mathbf{z}_2=0$.  $\tiny\begin{bmatrix}\mathbf{Z}_1\\\mathbf{Z}_2\end{bmatrix}$ is a complex Gaussian vector as each component is a linear combination of independent complex Gaussian random variables.  Furthermore, for a fixed $z$, $\mathbf{Z}_1\mathbf{Z}_2$ is a $2\times1$ matrix whose components are linear combinations of $A_jA_k$, $A_jB_k$, and $B_jB_k$, each of which we assume to have $0$ expectation.
Then
\begin{equation*}\Gamma_1=\begin{bmatrix}
\sum_{j=1}^n\alpha_j|z|^{2j}&0
\\0&\sum_{j=1}^m\beta_j|z|^{2j}
\end{bmatrix},\end{equation*}
\begin{equation*}\Gamma_{12}=\begin{bmatrix}
\sum_{j=1}^nj\alpha_j|z|^{2j}
\\\sum_{j=1}^mj\beta_j|z|^{2j}
\end{bmatrix},\end{equation*}
and $\Gamma_2=r_3$  It follows that
\begin{equation*}\Gamma=\begin{bmatrix}
\gamma_{11}&\gamma_{12}
\\ \gamma_{21} &\gamma_{22}
\end{bmatrix}=\frac{1}{r_3}\begin{bmatrix}
r_1&-r_{12}
\\-r_{12}&r_2
\end{bmatrix}.\end{equation*}
\begin{thm}[from \cite{Turin60}]
\label{thm:Turin}
Let $V=\begin{bmatrix}
v_1&\cdots&v_n
\end{bmatrix}'$ be such that $v_k=x_k+iy_k$ and $x_k,y_k$ are normally distributed with mean zero such that
\begin{equation}\label{pscov:eq}\mathbb{E}[x_jx_k]=\mathbb{E}[y_jy_k]\text{ and }\mathbb{E}[x_jy_k]=-\mathbb{E}[y_jx_k]\end{equation}
and $\Gamma=\mathbb{E}[VV^*]$, the covariance matrix.  Then for a Hermitian matrix $Q$, the Hermitian form $Y=V^*QV$ has characteristic function
\begin{equation*}\phi(t)=|I-it\Gamma Q|^{-1}.\end{equation*}
\end{thm}
Satisfying \eqref{pscov:eq} is equivalent to saying that the pseudo-covariance matrix, $\mathbb{E}[VV']$, is $0$.

Let $Q$ be the $2\times2$ diagonal matrix with $1$ and $-1$ on the diagonal and $V'=[zP'(z)\ zQ'(z)]$.  By Theorem \ref{thm:Turin} the characteristic function of $Y=|zP'(z)|^2-|zQ'(z)|^2$ is
\begin{equation}\label{eq:char}\begin{vmatrix}
1-it\gamma_{11}&it\gamma_{12}
\\-it\gamma_{12}&1+it\gamma_{22}
\end{vmatrix}^{-1}=\frac{1}{1+it(\gamma_{22}-\gamma_{11})+t^2(\gamma_{11}\gamma_{22}-\gamma_{12}^2)}\end{equation}
where $\Gamma=[\gamma_{jk}]$.  Then by the Fourier inversion formula we find the probability density function of $Y$ by integrating the product of \eqref{eq:char} and $e^{-iyt}/2\pi$ over $\mathbb{R}$ with respect to $t$.  The resulting integral can be computed using residues which gives the following expression for the probability density
\begin{equation*}\frac{1}{\sqrt{(\gamma_{22}+\gamma_{11})^2-4\gamma_{12}^2}}\exp\left\{\frac{-y(\gamma_{22}-\gamma_{11})-|y|\sqrt{(\gamma_{22}+\gamma_{11})^2-4\gamma_{12}^2}}{2(\gamma_{11}\gamma_{22}-\gamma_{12}^2)}\right\}.\end{equation*}
We then find the desired expectation (appearing inside the integral in \eqref{KacRice}) by integrating the product of this function with $|y|$ over the real line with respect to $y$, yielding
\begin{equation*}\frac{\gamma_{22}^2+\gamma_{11}^2-2\gamma_{12}^2}{\sqrt{(\gamma_{22}+\gamma_{11})^2-4\gamma_{12}^2}}.\end{equation*}
Using the values for $\Gamma$, this gives us
\begin{equation}\frac{r_1^2+r_2^2-2r_{12}^2}{r_3\sqrt{(r_1+r_2)^2-4r_{12}^2}}.\end{equation}

Multiplication by $\frac{1}{|z|^2}$ yields $\mathbb{E}[|J_H(z)|\big|H(z)=0]$.  Combining this with the fact that the p.d.f. of $H(z)$, a complex Gaussian with mean zero and covariance $r_3$, evaluated at zero is $\frac{1}{\pi r_3}$ completes the proof.
\end{proof}

\section{Proof of Theorems \ref{thm:refined} and \ref{thm:UV}}\label{sec:main}

Applying Theorem \ref{ThmGenLiWei} in the case of i.i.d.\@ coefficients, and integrating in polar coordinates $z=re^{i \theta}$ while making the change of variables $w = r^2$, $dw = 2r dr$, we obtain the following identity for the expected number of zeros
in an annulus $\Omega = \{ a<|z|^2<b \}$.

\be 
\EE N_H(a < |z|^2 < b) = \int_a^b \frac{1}{w} \frac{r_1^2+r_2^2-2r_{12}^2}{r_3^2\sqrt{(r_1+r_2)^2-4r_{12}^2}} dw,
\ee
where $r_1 = (a_n + a_m) c_n - b_n^2$, $r_2 = (a_n + a_m)c_m - b_m^2$, 
$r_{12} = b_n b_m$,
and $r_3 = a_n + a_m$, with
\be
a_k = \sum_{j=0}^k w^j, \quad
b_k = \sum_{j=1}^k jw^j, \quad
c_k = \sum_{j=0}^k j^2 w^j.
\ee

In particular, when $m=n$, we have
\be 
\EE N_H(a<|z|^2<b) = \int_{a}^b \frac{\sqrt{a_n^2 c_n^2-a_n c_n b_n^2}}{2 w a_n^2} dw.
\ee

Let us write this as
\be 
\EE N_H(a<|z|^2<b) = \int_{a}^b F_n(w) dw,
\ee
with
\be\label{eq:Fn}
F_n(w) = \frac{1}{2\sqrt{w}} \sqrt{\frac{c_n}{a_n}} \sqrt{\frac{a_n c_n - b_n^2}{w a_n^2}}.
\ee

With this change of variables, the conclusion of Theorem \ref{thm:refined} can be broken into the following three propositions.

\begin{prop}\label{prop:inside}
We have the following estimate as $n \rightarrow \infty$
\be
\int_{0}^{1+\frac{(\log n)^2}{n}} F_n(w) dw = O(n \log \log n).
\ee
\end{prop}

\begin{prop}\label{prop:middle}
We have the following estimate as $n \rightarrow \infty$
\be
\int_{1+\frac{(\log n)^2}{n}}^{1+(\log n)^{-1}} F_n(w) dw \sim \frac{1}{2} n \log n.
\ee
\end{prop}

\begin{prop}\label{prop:tail}
We have the following estimate as $n \rightarrow \infty$
\be
\int_{1+(\log n)^{-1}}^{\infty} F_n(w) dw = O(n \log \log n).
\ee
\end{prop}

\subsection{Preparatory lemmas}

We have \cite{Andy}
\be\label{eq:keyidentity}
\frac{a_n c_n - b_n^2}{w a_n^2} = \frac{1}{(w-1)^2}\left(1-\frac{(n+1)^2w^n}{a_n^2}\right),
\ee
which implies
\be\label{eq:est1}
\sqrt{\frac{a_n c_n - b_n^2}{w a_n^2}} \leq \frac{1}{|w-1|}.
\ee
We also have
\be \label{eq:est2}
\sqrt{\frac{c_n}{a_n}} \leq n,
\ee
which together with \eqref{eq:est1} implies
\be\label{eq:est3}
F_n(w) \leq \frac{1}{2\sqrt{w}} \frac{n}{|w-1|}.
\ee

\begin{lemma}\label{lemma:outside}
We have the following estimate
\be
\int_{2}^\infty F_n(w) dw \leq n.
\ee
\end{lemma}

\begin{proof}[Proof of Lemma \ref{lemma:outside}]
We use  \eqref{eq:est3} which implies,
for $w \geq 2$,
\be
F_n(w) \leq \frac{n}{2} (w-1)^{-3/2}.
\ee
Hence,
\be
\int_2^{\infty} F_n(w) dw \leq n,
\ee
as desired.
\end{proof}

\begin{lemma}\label{lemma:sliver}
We have the following estimate
\be
\int_{1+\frac{1}{n}}^{1+\frac{(\log n)^2}{n}} F_n(w) dw \leq 2 n \log \log n .
\ee
\end{lemma}

\begin{proof}[Proof of Lemma \ref{lemma:sliver}]
We use \eqref{eq:est3} which implies, for $w \geq 1$,
\be
F_n(w) \leq \frac{1}{2} \frac{n}{w-1}.
\ee
Hence,
\begin{align*}
\int_{1+\tfrac{1}{n}}^{1+\tfrac{(\log n)^2}{n}} F_n(w) dw &\leq n \log (w-1) \big\rvert_{w=1+\tfrac{1}{n}}^{w=1+\tfrac{(\log n)^2}{n}} \\
&= n \left( \log \left(\frac{(\log n)^2}{n}\right) - \log \left(\tfrac{1}{n}\right) \right) \\
&= 2 n \log \log n,
\end{align*} 
as desired.
\end{proof}

\begin{lemma}\label{lemma:inside}
We have the following estimate
\be
\int_0^{1- \frac{\log n}{n}} F_n(w) dw \leq (2+\sqrt{2})n.
\ee
\end{lemma}

\begin{proof}[Proof of Lemma \ref{lemma:inside}]
Over the interval $0 \leq w \leq \frac{1}{2}$, we use the estimate \eqref{eq:est3}, which implies
\be
F_n(w) \leq \frac{n}{\sqrt{w}}.
\ee
Hence,
\be
\int_0^{1/2} F_n(w) dw \leq \sqrt{2} n.
\ee
For the remaining interval $\tfrac{1}{2} \leq w \leq 1-\tfrac{\log n}{n}$,
we use the following estimate.

\noindent{\bf Claim.}
For $0 \leq w \leq 1-\tfrac{\log n}{n}$, we have
\be\label{eq:claim}
\sqrt{\frac{c_n}{a_n}} \leq \frac{2n}{\log n}.
\ee

To prove the Claim, we first show that $\frac{c_n}{a_n}$ is increasing for $w>0$.
We compute the derivative with respect to $w$.
\begin{align*}
\frac{d}{dw}\left[\frac{c_n}{a_n}\right]&=\frac{1}{a_n^2}\left(\displaystyle\sum_{j=1}^nj^3w^{j-1}\sum_{k=0}^nw^k-\sum_{j=1}^nj^2w^j\sum_{k=0}^nkw^{k-1}\right)
\\&= \frac{1}{a_n^2}\left(\displaystyle\sum_{j=1}^n\sum_{k=0}^nj^3w^{j+k-1}-\sum_{j=1}^n\sum_{k=0}^nj^2kw^{j+k-1}\right)
\\&= \frac{1}{a_n^2}\left(\displaystyle\sum_{j=1}^n\sum_{k=0}^nj^2(j-k)w^{j+k-1}\right)
\end{align*}
All terms in the sum for $k=0$ are positive, so we consider the sum from $k=1$ to $k=n$.
We have (using that the diagonal terms $j=k$ in the sum vanish)
\begin{align*}
\sum_{j=1}^n\sum_{k=1}^nj^2(j-k)w^{j+k-1}&=\sum_{j=2}^n\sum_{k=1}^{j-1}j^2(j-k)w^{j+k-1}+\sum_{j=1}^{n-1}\sum_{k=j+1}^nj^2(j-k)w^{j+k-1}
\\&=\sum_{j=2}^n\sum_{k=1}^{j-1}j^2(j-k)w^{j+k-1}-\sum_{k=2}^n\sum_{j=1}^{k-1}j^2(k-j)w^{j+k-1}
\\&=\sum_{j=2}^n\sum_{k=1}^{j-1}j^2(j-k)w^{j+k-1}-\sum_{j=2}^n\sum_{k=1}^{j-1}k^2(j-k)w^{j+k-1}
\\&=\sum_{j=2}^n\sum_{k=1}^{j-1}(j^2-k^2)(j-k)w^{j+k-1}>0.
\end{align*}
It follows that $\frac{c_n}{a_n}$ is an increasing function of $w$ for $w>0$.

This reduces proving the Claim to considering the value of $\frac{c_n}{a_n}$ at the right endpoint $w=1-\frac{\log n}{n}$ of the interval under consideration.
For this we use the identities
\be
a_n =\frac{1-w^{n+1}}{1-w},
\ee
\be
c_n = w \frac{w+1 - w^n\left[n^2(1-w)^2 + (2n+1)(1-w)\right]}{(1-w)^3},
\ee
which imply
\be
\frac{c_n}{a_n} = w \frac{w+1 - w^n\left[n^2(1-w)^2 + (2n+1)(1-w)\right]}{(1-w^{n+1})(1-w)^2},
\ee
or
\be\label{eq:cnan}
\frac{c_n}{a_n} = \frac{w}{1-w^{n+1}} \left( \frac{w+1}{(1-w)^2} -  w^n\frac{\left[n^2(1-w)^2 + (2n+1)(1-w)\right]}{(1-w)^2} \right).
\ee

Using the inequality $(1+x)^r \leq e^{r x}$ that holds for $x$ real and $r>0$, we have
\begin{align*}
w^{n} \big\rvert_{w=1-\frac{\log n}{n} } = \left(1-\frac{\log n}{n} \right)^{n} \leq \frac{1}{n},
\end{align*}
and
\be
\frac{w}{1-w^{n+1}} \big\rvert_{w=1-\frac{\log n}{n} } < \frac{n}{n-1}.
\ee
Applying this with \eqref{eq:cnan} gives
\begin{align*}
\frac{c_n}{a_n} \big\rvert_{w=1-\frac{\log n}{n}} &\leq \frac{n}{n-1}  \left(2-\frac{\log n}{n}\right) \frac{n^2}{(\log n)^2} \\
&\leq 4 \frac{n^2}{(\log n)^2} .
\end{align*}
This, along with the established monotonicity, verifies the statement in the Claim.

Applying the Claim, and using \eqref{eq:est1}, we have
\begin{align*}
\int_{\frac{1}{2}}^{1-\tfrac{\log n}{n}} F_n(w) dw &\leq \frac{2n}{\log n} \int_{\frac{1}{2}}^{1-\tfrac{\log n}{n}} \frac{1}{1-w} dw  \\
&=\frac{2n}{\log n}  \left( \log \left(\tfrac{1}{2}\right)
-\log \left(\frac{\log n}{n}\right)\right) \\
&\leq 2n.
\end{align*}
This completes the proof of the lemma.
\end{proof}

\begin{lemma}\label{lemma:uniform}
With $F_n(w)$ as in \eqref{eq:Fn}, we have
\be
F_n(w) = \frac{n}{2\sqrt{w}(w-1)} (1+o(1)), \quad (n \rightarrow \infty),
\ee
where the error term $o(1)$ converges to zero uniformly for $1+\frac{(\log n)^2}{n} < w < 2$.
\end{lemma}

\begin{proof}[Proof of Lemma \ref{lemma:uniform}]
For $1+\frac{(\log n)^2}{n} \leq w \leq 2$ we have (as $n \rightarrow \infty$)
\begin{align*}
\frac{(n+1)^2w^n}{a_n^2} &\leq (n+1)^2w^{-n} \\
&\leq (n+1)^2\left( 1+\frac{(\log n)^2}{n} \right)^{-n} \\
&= (n+1)^2 e^{-n \log \left( 1+ \frac{(\log n)^2}{n} \right)} \\
&= (n+1)^2 e^{-n \left(\frac{(\log n)^2}{n} + O\left( \frac{(\log n)^4}{n^2} \right) \right) }\\
&= (n+1)^2 e^{-(\log n)^2 + O\left( \frac{(\log n)^4}{n} \right) }\\
&= (n+1)^2 e^{-(\log n)^2}\left( 1 + O\left( \frac{(\log n)^4}{n} \right)\right)  \\
&=o(1).
\end{align*}
In view of \eqref{eq:keyidentity}, this implies
\be\label{eq:factor2}
\frac{a_n c_n - b_n^2}{w a_n^2} = 1+o(1), \quad (n \rightarrow \infty).
\ee

\noindent {\bf Claim.} We have, uniformly for $1+\frac{(\log n)^2}{n} \leq w \leq 2$,
\be\label{eq:factor1}
\sqrt{\frac{c_n}{a_n}} = n\left(1+O\left(\frac{1}{(\log n)^2}\right)\right), \quad (n \rightarrow \infty).
\ee

For the proof of the Claim, we use the identities
\be\label{eq:an}
a_n =\frac{w^{n+1}-1}{w-1},
\ee
\be\label{eq:cn}
c_n = w \frac{ w^n\left[n^2(w-1)^2 - (2n+1)(w-1)\right] - (w+1)}{(w-1)^3},
\ee
which imply
\begin{align*}
\frac{c_n}{a_n} &= w^{n+1} \frac{ \left[n^2(w-1)^2 - (2n+1)(w-1)\right]-\frac{w+1}{w^n}}{(w^{n+1}-1)(w-1)^2} \\
&= \frac{ \left[n^2(w-1)^2 - (2n+1)(w-1)\right]-(w+1)w^{-n}}{(1-w^{-(n+1)})(w-1)^2}.
\end{align*}

We have, uniformly for $1+\frac{(\log n)^2}{n} \leq w \leq 2$,
\begin{align*}
w^{-n} &\leq \left(1+\frac{(\log n)^2}{n} \right)^{-n}\\
&= e^{-n \log \left(1+\frac{(\log n)^2}{n} \right)} \\
&= e^{-n \left( \frac{(\log n)^2}{n} + O\left( \frac{(\log n)^4}{n^2} \right) \right)}\\
&= e^{- (\log n)^2} \left(1 + O\left( \frac{(\log n)^4}{n} \right) \right).
\end{align*}
This implies
\be\label{eq:wn}
w^{-n} = O\left( e^{- (\log n)^2} \right),
\ee
and 
\be
1-w^{-n-1} = 1+O\left( e^{- (\log n)^2} \right),
\ee
and hence we have
\be\label{eq:ratiosep}
\frac{c_n}{a_n} = \left( n^2 -\frac{(2n+1)}{w-1} - \frac{(w+1)w^{-n}}{(w-1)^2} \right) \left( 1+O\left( e^{- (\log n)^2} \right)\right).
\ee
We have, uniformly for $1+\frac{(\log n)^2}{n} \leq w \leq 2$,
\be
\frac{1}{w-1} \leq \frac{n}{(\log n)^2}
\ee
which implies
\be
\frac{2n+1}{w-1} = O\left( \frac{n^2}{(\log n)^2} \right),
\ee
and (using \eqref{eq:wn} as well)
\begin{align*}
\frac{(w+1)w^{-n}}{(w-1)^2} &= O\left( \frac{ n^2 e^{-(\log n)^2}}{(\log n)^4} \right)  \\
&= O\left( \frac{ n^2}{(\log n)^2}  \right).
\end{align*}
Applying this in \eqref{eq:ratiosep} gives
\be
\frac{c_n}{a_n} = n^2  \left( 1+O\left( \frac{1}{(\log n)^2} \right)\right).
\ee
This verifies the statement in the Claim.

The Claim, together with \eqref{eq:factor2}, implies
\be
F_n(w) = \frac{n}{2\sqrt{w}(w-1)} (1+o(1)), \quad (n \rightarrow \infty),
\ee
where the convergence is uniform over the interval
$I_n:=[1+\frac{(\log n)^2}{n},2]$,
and this completes the proof of the lemma.
\end{proof}

\subsection{Proofs of Propositions \ref{prop:inside}, \ref{prop:middle}, and \ref{prop:tail}}

\begin{proof}[Proof of Proposition \ref{prop:inside}]
The Lemmas \ref{lemma:inside} and \ref{lemma:sliver} take care of the intervals $0 \leq w \leq 1-\frac{\log n}{n}$ and $1+\frac{1}{n} \leq w \leq 1+\frac{(\log n)^2}{n}$.
On the interval $ 1-\frac{1}{n} \leq w \leq  1 + \frac{1}{n}$, we use the estimate
\be
F_n(w) \leq n^2
\ee
in order to conclude
\be\label{eq:near1}
\int_{1 - \tfrac{1}{n}}^{1+\tfrac{1}{n}} F_n(w) dw \leq 2n.
\ee
On the remaining interval $ 1 - \tfrac{\log n}{n} \leq w \leq 1-\frac{1}{n} $, we use the estimate \eqref{eq:est3}
which implies (since $\tfrac{1}{\sqrt{w}} \leq 2$ on this interval)
\be
F_n(w) \leq \frac{n}{1-w}.
\ee
Hence,
\begin{align*}
\int_{1-\tfrac{\log n}{n}}^{1-\tfrac{1}{n}} F_n(w) dw &\leq -n \log (1-w) \bigg\rvert_{w=1-\tfrac{\log n}{n}}^{w=1-\tfrac{1}{n}} \\
&= n \left( \log \left(\frac{\log n}{n}\right) - \log \left(\tfrac{1}{n}\right) \right) \\
&= n \log \log n.
\end{align*}
Together with \eqref{eq:near1} and the estimates in Lemmas \ref{lemma:inside} and \ref{lemma:sliver}, this implies
\be
\int_0^{1+\tfrac{(\log n)^2}{n}} F_n(w) dw \leq (4+\sqrt{2})n + 3 n \log \log n,
\ee
which implies the statement in the proposition.
\end{proof}

\begin{proof}[Proof of Proposition \ref{prop:middle}]
Let $I_n := [1+\frac{(\log n)^2}{n}, 1+ \frac{1}{\log n}]$. Applying Lemma \ref{lemma:uniform}, we have (as $n \rightarrow \infty$)
\be\label{eq:mainasymp1}
\int_{I_n} \frac{\sqrt{a_n^2 c_n^2-a_n c_n b_n^2}}{2 w a_n^2} dw = (1+o(1)) \int_{I_n} \frac{n}{2\sqrt{w}(w-1)} dw,
\ee
where we have relied at this step on the uniformity of the $o(1)$ estimate.
We use the change of variables $w=r^2$, $dw = 2r dr$.
Then
\begin{align*}
    \int_{I_n} \frac{1}{2\sqrt{w}(w-1)} dw &= \int_{I_n'}  \frac{1}{r^2-1} dr \\
    &= \frac{1}{2}\int_{I_n'}  \frac{1}{r-1} - \frac{1}{r+1} dr \\
    &= \frac{1}{2} \left(
    \log (r-1) - \log (r+1) \right) \bigg\rvert_{\sqrt{1+\frac{(\log n)^2}{n}}}^{\sqrt{1+\frac{1}{\log n}}} \\
    &= - \frac{1}{2}
    \log \left(\frac{(\log n)^2}{n}\right) + O(1) \\
     &= \frac{1}{2}
    \log n - \log \log n+ O(1)\\
    &= \left(\frac{1}{2}
    \log n \right)(1+o(1)) .
\end{align*}
Together with \eqref{eq:mainasymp1}, this proves the lemma.
\end{proof}

\begin{proof}[Proof of Proposition \ref{prop:tail}]
In view of Lemma \ref{lemma:outside}, it suffices to show
\be\label{eq:near2}
\int_{1+\frac{1}{\log n}}^{2} F_n(w) dw = O(n \log \log n), \quad (n \rightarrow \infty).
\ee
 Applying Lemma \ref{lemma:uniform}, we have (as $n \rightarrow \infty$)
\be\label{eq:mainasymp2}
\int_{I_n} F_n(w) dw = (1+o(1)) \int_{I_n} \frac{n}{2\sqrt{w}(w-1)} dw,
\ee
where we have relied at this step on the uniformity of the $o(1)$ estimate.
Using again the change of variables $w=r^2$, $dw = 2r dr$, we have
\begin{align*}
    \int_{I_n} \frac{1}{2\sqrt{w}(w-1)} dw 
    &= \frac{1}{2} \left(
    \log (r-1) - \log (r+1) \right) \bigg\rvert_{\sqrt{1 + \frac{1}{\log n}}}^{\sqrt{2}} \\
    &= - \frac{1}{2}
    \log \left(\frac{1}{\log n}\right) + O(1) \\
    &= \left(\frac{1}{2}
    \log \log n \right)(1+o(1)) .
\end{align*}
Together with \eqref{eq:mainasymp2}, this implies \eqref{eq:near2} and completes the proof of the proposition.
\end{proof}

\subsection{Proof of Theorem \ref{thm:UV}}
Let $V$ be a subset of the unit disc whose closure is also contained in the unit disc, i.e., $V$ is contained in $|z| \leq R$ for some $0<R<1$. We have 
\begin{equation*}
\mathbb{E}N_H(V)=\frac{1}{\pi}\int_VF_n(|z|^2)\,dA(z),
\end{equation*}
where 
\be\label{eq:Fnz}
F_n(|z|^2) = \frac{1}{2|z|^2} \sqrt{\frac{c_n(|z|^2)}{a_n(|z|^2)}} \sqrt{\frac{a_n(|z|^2) c_n(|z|^2) - b_n(|z|^2)^2}{a_n(|z|^2)^2}},
\ee
with
\begin{align*}
a_n(|z|^2)=\sum_{j=0}^n|z|^{2j}&=\frac{1-|z|^{2n+2}}{1-|z|^2}
\\b_n(|z|^2)=\sum_{j=0}^nj|z|^{2j}&=|z|^2\frac{n|z|^{2n+2}-(n+1)|z|^{2n}+1}{(1-|z|^2)^2}
\\c_n(|z|^2)=\sum_{j=0}^nj^2|z|^{2j}&=|z|^2\frac{-n^2|z|^{2n+4}+(2n^2+2n-1)|z|^{2n+2}-(n+1)^2|z|^{2n}+|z|^2+1}{(1-|z|^2)^3}.
\end{align*}
for $|z|\ne1$. 
We compute that for $|z|<1$
\be\label{eq:limitFn}
\lim_{n\to\infty}F_n(|z|^2)=\frac{1}{2}\frac{\sqrt{|z|^2+1}}{(1-|z|^2)^2}
\ee
which is clearly integrable on $V$.  Moreover, this limit converges uniformly for $|z|<R$, since $a_n(w)$, $b_n(w)$, and $c_n(w)$ are the Taylor polynomials for
\begin{equation*}\frac{1}{1-w},\quad\frac{w}{(1-w)^2},\text{ and }\frac{w^2+w}{(1-w)^3}\end{equation*}
respectively which have radius of convergence $1$. 
Since  $a_nc_n-b_n^2>0$ and $a_n \geq 1 > 0$, we have uniform convergence of the integrand on $|z| \leq R$, so that the limit can be taken inside the integral and using \eqref{eq:limitFn} we obtain
\begin{equation*}
\lim_{n\to\infty}\mathbb{E}N_H(V)=\frac{1}{2\pi}\int_V\frac{\sqrt{1+|z|^2}}{(1-|z|^2)^2} dA(z)=C_V,
\end{equation*}
which proves the first part of the theorem (while also providing a formula for $C_V$).

Let $U$ be an open set whose closure is contained in the exerior of the closed unit disk, equation (\ref{eq:est3}) implies that
$F_n(|z|^2)/n$
is bounded by an integrable function for all $n$ and thus we may take the limit inside the integral, which gives
\begin{align*}
\lim_{n\to\infty}\frac{\mathbb{E}N_H(U)}{n}&=\lim_{n\to\infty}\int_U\frac{1}{2\pi|z|^2}\frac{\sqrt{a_nc_n(a_nc_n-b_n^2)}}{na_n^2}\,dA(z)
\\&=\int_U\lim_{n\to\infty}\frac{1}{2\pi|z|^2}\frac{\sqrt{a_nc_n(a_nc_n-b_n^2)}}{na_n^2}\,dA(z)
\\&=\frac{1}{2\pi}\int_U\frac{1}{|z|(|z|^2-1)}\,dA(z)=C_U,
\end{align*}
where we have omitted the details of the computation of the limit in the last line which is again elementary.
This proves the theorem (while also providing a formula for $C_U$).

\section{Proof of Theorems \ref{thm:existence} and \ref{thm:UB}}
\label{sec:existenceproof}

\begin{proof}[Proof of Theorem \ref{thm:existence}]
We use an adaptation of the idea from \cite{LundRand}, except in the current setting the leading coefficient in $p$ will taken to be random and non-perturbative.

Let us sample 
$$q(z) = \sum_{k=0}^{n-1} e^{i \theta_k} z^k,$$
with i.i.d.\@ coefficients $e^{i \theta_k}$ generated by angles $\theta_k$ sampled uniformly from the interval $[0,2\pi]$.
We then take 
$p(z) = \beta z^n + q(z)$,
with $\beta$ being a $\pm 1$ Bernoulli random variable with equal probability of taking the values $1$ and $-1$ (in fact $\beta$ can be removed by multiplying $p(z) + \ol{q(z)}$ by the sign of $\beta$ that does not alter the number of zeros, but we include it in order to satisfy the hypothesis of Lemma \ref{lemma:expected} below).  Note that, unlike in the previous model, in this ad hoc model $p$ and $q$ are highly dependent.

Taking real and imaginary parts of the equation $p(z) + \ol{q(z)}=0$ then gives the following system of equations.
\be\label{eq:system}
\begin{cases}
\beta \Re\{ z^n \} + 2\Re \{q(z)\} = 0, \\
\Im \{ z^n \}  = 0.
\end{cases}
\ee
The second equation in \eqref{eq:system} describes $n$ lines passing through the origin, and the solutions of the system \eqref{eq:system} correspond to the zeros of the univariate polynomials obtained from restricting the first equation to each of these lines.

Writing $z=r e^{i \theta_j}$ with $r \in \RR$ and $\theta_j = \frac{j \pi}{n}$ for $j=0,1,2,...,n-1$, we obtain, after substitution into the first equation in \eqref{eq:system},
\be\label{eq:polar2}
\beta r^n + \sum_{k=0}^{n-1} 2 \cos(k \theta_j + \theta_k) r^k = 0.
\ee
We notice that the random angle $k \theta_j - \theta_k$ is uniformly distributed (up to addition of a multiple of $2\pi$) in the interval $[0,2 \pi]$, and hence the random coefficients $2 \cos(k \theta_j + \theta_k)$ are identically distributed and have (by symmetry) zero mean.  Moreover, by a simple elementary computation we find that they have variance one.

Since the event $p(0)+\ol{q(0)} = 0$ has zero probability, we have that the total number $N$ of solutions of the system \eqref{eq:system} is almost surely given by 
\be\label{eq:Nsum}
N = \sum_{j=0}^{n-1} N_j,
\ee 
where $N_j$ denotes the number of solutions of \eqref{eq:polar2}. 
By linearity of expectation
\begin{align*}
\EE N &= \sum_{j=0}^{n-1} \EE N_j \\
      &= n \EE N_0,
\end{align*}
where the second line follows from the fact that the left hand sides of \eqref{eq:polar2} are identically distributed distributed (as $j$ varies).

Notice that the number $N_0$ of real zeros of the polynomial $f(x) = \sum_{k=0}^{n} \beta_{k} x^k$ (with $\beta_k := 2\cos \theta_k$) is equivalent to the number of real zeros of the transformed polynomial $g(x) = x^n f(x^{-1}) = \sum_{k=0}^{n} \beta_{n-k} x^k$.  Next we will use the following result that is a special case of \cite[Thm. 1.4]{DoNguyenVu}.

\begin{lemma}\label{lemma:expected}
Let $\alpha_k$ be a sequence of independent random variables with zero mean, unit variance, and uniformly bounded $(2+\e)$-moments.
Let $g(x) = \sum_{k=0}^{n} \alpha_k x^k$.
\be
\EE \, \# \{ x\in \RR : g(x)=0\} = \frac{2}{\pi} \log n + O(1), \quad (n \rightarrow \infty).
\ee
\end{lemma}

Since the coefficients $\alpha_k = \beta_{n-k}$ satisfy the conditions in Lemma \ref{lemma:expected}, we conclude that $\EE N_0 \sim \frac{2}{\pi} \log n$ as $n \rightarrow \infty$, and this gives
\be
\EE N = \frac{2}{\pi} n \log n + O(n), \quad (n \rightarrow \infty),
\ee
Since there is at least one instance in parameter space attaining the average, this 
proves the theorem.
\end{proof}

\begin{proof}[Proof of Theorem \ref{thm:UB}]
As in the proof of Theorem \ref{thm:existence},
taking real and imaginary parts of the equation $p(z) + \ol{q(z)}=0$ gives the system of equations
\be\label{eq:systemLittlewood}
\begin{cases}
\beta \Re\{ z^n \} + 2\Re \{q(z)\} = 0, \\
\Im \{ z^n \}  = 0,
\end{cases}
\ee
and the solutions of the system \eqref{eq:systemLittlewood} correspond to the zeros of the univariate polynomials obtained from restricting the first equation to each of the lines from the solution set of the second equation.
Writing $z=r e^{i \theta_j}$ with $r \in \RR$ and $\theta_j = \frac{j \pi}{n}$ for $j=0,1,2,...,n-1$, we obtain, after substitution into the first equation in \eqref{eq:system},
\be\label{eq:polar3}
r^n + \sum_{k=0}^{n-1} 2 \cos(k \theta_j + \theta_k) r^k = 0.
\ee
Since the constant coefficient is $2$ and the remaining coefficients are in absolute value at most $2$, by \cite[Thm. 4.1]{BEK}, there exists $c_1>0$ such that \eqref{eq:polar3} has at most $c_1 \sqrt{n}$ solutions in $\RR \setminus [-1,1]$.

Making the substitution $r=2x$ in \eqref{eq:polar3} we obtain
\be\label{eq:subst}
x^n + \sum_{k=0}^{n-1} 2^{k+1-n} \cos(k \theta_j + \theta_k) x^k = 0.
\ee
Noting that the leading coefficient equals $1$ while the remaining coefficients $2^{k+1-n} \cos(k \theta_j + \theta_k)$ are each at most $1$ in absolute value, by \cite[Thm. 4.1]{BEK}, there exists $c_2>0$ such that \eqref{eq:subst} has at most $c_2 \sqrt{n}$ solutions in $\RR \setminus [-1,1]$, and hence \eqref{eq:polar3} has at most $c \sqrt{n}$ solutions in  $\RR \setminus [-1/2,1/2]$ and in partiular at most this many solutions in $\RR \setminus [-1,1]$.

We conclude that \eqref{eq:subst} has at most $2c\sqrt{n}$ real solutions for each $k=0,1,2,...,n-1$, and hence with $C=c_1+c_2$ the system \eqref{eq:polar3} has at most $C \cdot n \sqrt{n}$ solutions, as desired.
\end{proof}

\section{Concluding Remarks}\label{sec:concl}

\subsection*{Intensity of zeros inside the unit disk} The proof of Theorem \ref{thm:UV} provided formulas for the constants, $C_U, C_V$.  In particular, the average number of zeros of $H(z) = p_n(z) + \ol{q_n(z)}$ in a region $V$ whose closure is contained in the open unit disk converges to
$$ C_V = \int_V \frac{1}{2\pi} \frac{\sqrt{1+|z|^2}}{(1-|z|^2)^2} dA(z).$$
Viewing the zeros of $H$ as a point process, the above integrand is the limit of the so-called \emph{first intensity}.  It is interesting that we have the following estimate
\be
 \frac{1}{2\pi} \frac{\sqrt{1+|z|^2}}{(1-|z|^2)^2} < \frac{1}{\pi} \frac{1}{(1-|z|^2)^2}, \quad 0<|z|<1,
\ee
showing that inside the unit disk the limiting first intensity for the harmonic case is dominated by that of the analytic case \cite{EdelmanKostlan95}, \cite{SodinZeros}.  To reiterate, the addition of the anti-analytic term $\ol{q(z)}$ not only has a much more dramatic effect in the exterior of the unit disk $|z|>1$ (as indicated by the order-$n$ asymptotic growth in the statement of Theorem \ref{thm:UV}) than in the interior, it actually \emph{diminishes} the limiting first intensity inside the unit disk,
see Fig. \ref{fig:intensity} for a depiction of this at degree $n=m=30$.

\begin{figure}[h]
\centering
\includegraphics[scale=0.5]{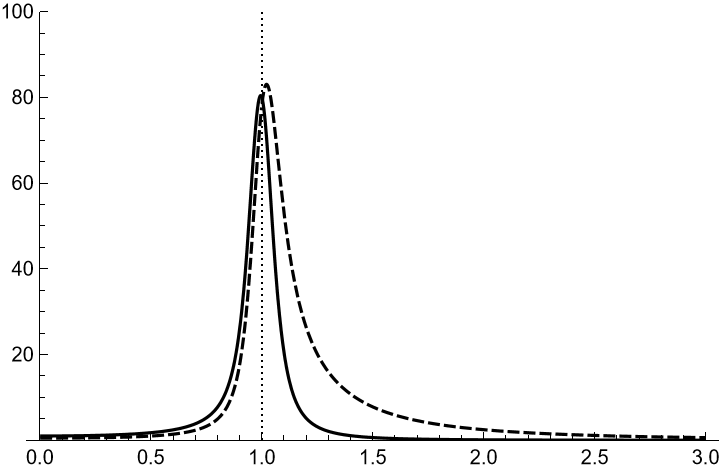}
\hspace{0.1in}
\includegraphics[scale=0.5]{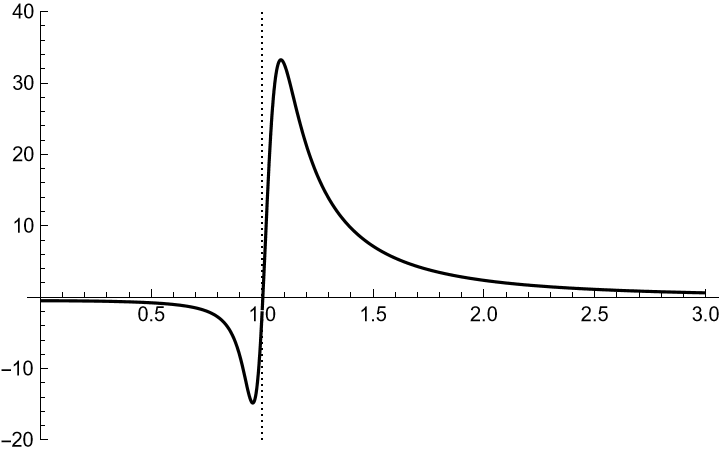}
\caption{Left: The first intensity with respect to radius plotted for $n=m=30$ for the analytic case (solid) and the harmonic case (dotted).
Right: The difference between the first intensities (harmonic minus analytic) for $n=m=30$, showing a noticeable abrupt transition as the radius crosses unity.}
\label{fig:intensity}
\end{figure}

\subsection*{The regime $m = \alpha n$, $0<\alpha<1$}
As mentioned in the introduction, the second named author showed in his thesis \cite{AndyThesis} that the average number of zeros of $H_{n,m}(z) = p_n(z) + \ol{q_m(z)}$ is asymptotic to $n$ when $m$ is fixed and $n \rightarrow \infty$.  From Theorem \ref{thm:iid} we have that the asymptotic average gains a logarithmic factor when $m = n$.  It is natural to ask about the intermediate regime, say $m = \alpha n$, $0<\alpha<1$.
Since $q_m(z)$ may be viewed as an independent truncated copy of $p_n(z)$, one might be tempted to guess, by analogy with the results in \cite{Lerariotruncated}, \cite{AndyZach}, that the average number of zeros is asymptotic to $\frac{1}{2} m \log m = \frac{\alpha}{2} n \log (\alpha n) \sim \frac{\alpha}{2} n \log n$ as $n \rightarrow \infty$.  However, unlike in the case of the Kostlan or Weyl ensembles, for the Kac ensemble an independent copy of $p_n(z)$ truncated to degree $m = \alpha n$ is negligibly small in comparison with $p_n(z)$ throughout the most important region $|z|>1+ \frac{1}{n}$, and this leads us to conjecture that in the case $m = \alpha n$ the average number of zeros is asymptotically proportional to $n$ (hence resembling the case when $m$ is fixed) with the constant of proportionality perhaps depending on $m$.

\subsection*{Asymptotic expansion of the expected number of zeros}

Theorem \ref{thm:iid} provides the precise leading order asymptotic for the average number of zeros when $m=n \rightarrow \infty$.  Obtaining additional terms in an asymptotic expansion seems challenging.  
One could attempt to determine additional terms in the asymptotic expansion by further refining the method used in our proof of Theorem \ref{thm:refined}.  An alternative approach, perhaps with some hope of obtaining a complete asymptotic expansion, is to devise a uniformly converging asymptotic representation of the integrand.  The particular form of the integrand's nondominated convergence suggests that this will require a multi-scale approach as is commonly used in the method of matched asymptotic expansions coming from boundary layer theory \cite{Lagerstrom}, where the choice of ``boundary layers'' may be expected to resemble the subdivision of intervals used in the proof of Theorem \ref{thm:refined}.

In the known asymptotic expansion \cite{Wilkins} for the average number of real zeros of a random Kac polynomial of degree $n$, the logarithm only appears in the leading order term, and the remaining terms are (constant multiples of) integer powers of $n$.
A na\"ive guess is that the same holds for the number of complex zeros of a harmonic Kac polynomial $H(z) = p_n(z) + \ol{q_n(z)}$,
and this leads us to ask whether $\EE N_{H} (\CC)$ admits a complete asymptotic expansion of the form
$$ \EE N_{H} (\CC) \sim \frac{1}{2} n \log n + c_1 n + c_0 + c_{-1} n^{-1} + c_{-2} n^{-2} + \cdots $$
with constant coefficients $c_k$, for $k=1,0,-1,-2,...$



\subsection*{Universality}
One obvious direction for future research is prompted by the fact that all of the previous results on complex zeros of harmonic polynomials are confined to Gaussian models (discounting the non-Gaussian model used here in the proof of Theorem \ref{thm:existence} since it is customized for the particular purpose of the existence proof and not suited for studying harmonic polynomials in a broad setting).  In stating Conjecture \ref{conj:iid} we have already indicated one specific problem in this direction.  More broadly, for each of the Gaussian models mentioned in the survey of literature provided Section \ref{sec:models}, one may ask what happens when the coefficients are independent but not Gaussian (while retaining the same sequence of variances).
Based on robust results for a variety of models of random analytic polynomials showing local universality for the number of real zeros as well as the number of complex zeros, one may expect that the same asymptotics found above as well as in \cite{LiWei}, \cite{Lerariotruncated}, \cite{AndyZach}, \cite{AndyThesis} continue to hold in non-Gaussian settings (under some control on the moments of the distribution).

\subsection*{Hengartner's valence problem for logharmonic polynomials}
As shown in \cite{LundRand} as well as Theorem \ref{thm:existence} of the current paper, in addition to providing a broad complementary perspective, probabilistic studies of harnomic polynomials can lead to new insights on the deterministic side.  Indirect probabilistic methods surely hold further potential for the study of related extremal problems.  For instance, instead of considering the sum $p(z) + \ol{q(z)}$, one may ask about the maximal valence (number of preimages of a prescribed point in the complex plane) of the product $p(z) \ol{q(z)}$, which is referred to as a \emph{logharmonic} polynomial (the valence is finite if $p$ is not a constant multiple of $q$).  This problem was posed by Hengartner \cite{BL} (see also \cite{BH}, \cite{AH}).  Recent progress \cite{KLP} estimating the valence of logharmonic polynomials includes proof of a Bshouty and Hengartner's conjectured \cite{BH} upper bound of $3n-1$ for the case when $q$ is linear, and an improvement on the upper bound provided by Bezout's theorem for each $m,n$, but there is a lack of illuminating examples in this setting, and it is still not known whether the maximal valence increases quadratically in $n,m$, or indeed whether it even increases faster than linearly in $n,m$.  Probabilistic methods may be promising here for showing the existence of examples with an abundance of zeros, say for $m=n$.  Hence, it seems an enticing open problem to obtain asymptotics, or even lower bounds, for the average number of zeros when $p,q$ are sampled from one of the models discussed in Section \ref{sec:models} above (or from another model chosen strategically to generate many zeros on average).

\subsection*{Critical lemniscates}

The critical set (vanishing set of the Jacobian) of $H(z) = p(z) + \ol{q(z)}$ is the zero set $|p'(z)|^2-|q'(z)|^2=0$ or equivalently the level set $\left| \frac{p'(z)}{q'(z)} \right| =1$, a rational lemniscate. This set generically consists of at most $n-1$ smooth connected components, each diffeomorphic to a circle. As $z$ crosses the critical set the orientation of the mapping changes sign, an important example of behavior of harmonic maps that is not exhibited by analytic maps.
The topological complexity of the critical set (number of connected components) marks how wildly the mapping changes orientation and hence serves as one indicator for how far the mapping deviates from analytic behavior.  The problem of studying the critical lemniscates of random harmonic polynomials has been posed in \cite{Lerariotruncated}.
In the case when $p,q$ have i.i.d.\@ coefficients we suspect, partially guided by the analysis of polynomial lemniscates associated to Kac polynomials in \cite{KoushikErik}, that as $m=n \rightarrow \infty$ the average number of connected components of the critical set is asymptotically proportional to $n$.

\subsection*{Gravitational lensing by point masses with uniformly distributed locations}

As we have mentioned in the introduction,
the current setting of zeros of random harmonic polynomials runs parallel to some central problems in the theory of gravitational lensing, namely, in stochastic microlensing.
For instance, consideration of gravitational lensing by point masses leads to the lensing equation
\be
z - \sum_{k=1}^n \frac{m}{\ol{z} - \ol{z_k}} = 0,
\ee
where we take $n$ equal (for simplicity) point masses with mass $m$ at random positions $z_k \in \CC$ in the lens plane.  When the positions $z_k$ are sampled independently and uniformly from a disk, and the mass $m$ is fixed, the number of solutions (which represent lensed images of a background source) is shown to be asymptotically $n$ in \cite{PettersStochastic2}.  This outcome is rather deficient as it is asymptotically the fewest lensed images possible (from Morse theory it follows that there are at least $n+1$ lensed images \cite{Petters1}).
On the other hand, according to numerical experiments \cite{SeanThesis}, if instead of taking $m$ fixed we take the mass $m = c/n$ to be inversely proportional to $n$ then the average number of zeros appears to be asymptotically $\lambda n$ with $\lambda > 1$.
It seems a worthwhile open problem to try to show this analytically, which reduces to asymptotic analysis of the Kac-Rice type integral obtained in \cite{PettersStochastic2}.


\bibliographystyle{abbrv}
\bibliography{RandHarmonic}

\end{document}